\documentclass{amsart2010}
\usepackage{graphicx}
\usepackage{amsmath}
\usepackage{amsfonts}
\usepackage{amssymb}
\usepackage[all]{xy}

\newtheorem{thm}{Theorem}
\newtheorem{cor}[thm]{Corollary}
\newtheorem{lem}[thm]{Lemma}
\newtheorem{prop}[thm]{Proposition}
\theoremstyle{definition}

\newtheorem*{theorem*}{Theorem}

\theoremstyle{remark}

\numberwithin{equation}{section}
\newcommand{\norm}[1]{\left\Vert#1\right\Vert}

\newcommand{\To}{\longrightarrow}

\def\<{\langle}
\def\>{\rangle}
\begin{document}
\title[]{Universal composition operators}
\author{Jo\~a{o} R. Carmo and  S. Waleed Noor } 
\address{IMECC, Universidade Estadual de Campinas, Campinas-SP, Brazil.}
\email{$\mathrm{joao.mr2@hotmail.com}$(1st author)$,\mathrm{waleed@unicamp.br}$(2nd author).}

\begin{abstract} A Hilbert space operator $U$ is called \emph{universal} (in the sense of Rota) if every Hilbert space operator is similar to a multiple of $U$ restricted to one of its invariant subspaces. It follows that the \emph{Invariant Subspace Problem} for Hilbert spaces is equivalent to the statement that all minimal invariant subspaces for $U$ are one dimensional. In this article we characterize all linear fractional composition operators $C_{\phi} f=f\circ\phi$ that have universal translates on both the classical Hardy spaces $H^2(\mathbb{C}_+)$ and $H^2(\mathbb{D})$ of the half-plane and the unit disk respectively. The surprising new example is the composition operator on $H^2(\mathbb{D})$ with \emph{affine} symbol $\phi_a(z)=az+(1-a)$ for $0<a<1$. This leads to strong characterizations of minimal invariant subspaces and eigenvectors of $C_{\phi_a}$ and offers an alternative approach to the ISP.
\end{abstract}
\maketitle{}
\section{Introduction} One of the most important open problems in operator theory is the \emph{Invariant Subspace Problem} (ISP), which asks: Given a complex separable Hilbert space $\mathcal{H}$ and a bounded linear operator $T$ on $\mathcal{H}$, does $T$ have a nontrivial invariant subspace? An invariant subspace of $T$ is a closed subspace $E\subset\mathcal{H}$ such that $TE\subset E$. The recent monograph by Chalendar and Partington \cite{Chalendar Partington book} is a reference for some modern approaches to the ISP. Rota \cite{Rota} in 1960 demonstrated the existence of operators that have invariant subspace lattices so rich that they could model \emph{every} Hilbert space operator. \\ \\
$\mathbf{Definition.}$
\emph{Let $\mathcal{B}$ be a Banach space and $U$ a bounded linear operator on $\mathcal{B}$. Then $U$ is said to be universal for $\mathcal{B}$, if for any bounded linear operator $T$ on $\mathcal{B}$ there exists a constant $\alpha\neq 0$ and an invariant subspace $\mathcal{M}$ for $U$ such that the restriction $U|_\mathcal{M}$ is similar to $\alpha T$. } \\ 

If $U$ is universal for a separable, infinite dimensional Hilbert space $\mathcal{H}$, then the ISP is equivalent to the assertion that every infinite dimensional invariant subspace for $U$ contains a nontrivial proper invariant subspace, or equivalently, the \emph{minimal} invariant subspaces for $U$ are all one dimensional. The main tool thus far  for identifying universal operators has been the following criterion of Caradus \cite{Caradus}. \\ \\
$\mathbf{The \ Caradus \  criterion.}$
\emph{	Let $\mathcal{H}$ be a separable infinite dimensional Hilbert space and $U$ a bounded linear operator on $\mathcal{H}$. If $\mathrm{ker}(U)$ is infinite dimensional and $U$ is surjective, then $U$ is universal for $\mathcal{H}$. } \\

If $X$ is a Banach space of holomorphic functions on an open set $U\subset\mathbb{C}$ and if $\phi$ is a holomorphic self-map of $U$, the \emph{composition operator} with \emph{symbol} $\phi$ is defined by $C_\phi f=f\circ\phi$ for any $f\in X$. The study of composition operators is concerned with  the comparison of  properties of $C_\phi$ with those of the symbol $\phi$. If $X$ is the Hardy space $H^2(\mathbb{D})$ of the open unit disk then every self-map $\phi$ of $\mathbb{D}$ induces a bounded $C_\phi$.  In contrast if $X$ is the Hardy space $H^2(\mathbb{C}_+)$ of the right half-plane, then a holomorphic self-map $\psi$ of $\mathbb{C}_+$ induces a bounded $C_\psi$ if and only if $\psi$ has a finite angular derivative at the fixed point $\infty$. That is, if $\psi(\infty)=\infty$ and if the non-tangential limit
\begin{equation}\label{Ang. Der. at infinity}
\psi'(\infty):=\lim_{w\to\infty}\frac{w}{\psi(w)}
\end{equation}
exists and is finite. This was proved by Matache in \cite{Matache weighted}.

 Nordgren, Rosenthal and Wintrobe \cite{Nordgren-Rosenthal-Wintrobe} gave a remarkable reformulation of the ISP in terms of composition operators on $H^2(\mathbb{D})$. They showed that if $\phi$ is a \emph{hyperbolic automorphism} of $\mathbb{D}$ (having two distinct fixed points on $\mathbb{T}:=\partial\mathbb{D}$), then $C_\phi-\lambda I$ is universal for all $\lambda$ in the interior of the point spectrum of $C_\phi$. Since $C_\phi-\lambda I$ and $C_\phi$ have the same invariant subspaces, the ISP has a positive solution if and only if the minimal non-trivial invariant subspaces of $C_\phi$ are all one dimensional. The last three decades have seen many works dedicated to this approach to the ISP (see \cite{Chkliar},\cite{Gallardo-Gorkin},\cite{Matache minimal},\cite{Matache eigenfunction},\cite{Mortini},\cite{Nordgren-Rosenthal-Wintrobe},\cite{Shapiro redux}). They all focus on the hyperbolic automorphism
\begin{equation}\label{Canonical Auto}
h(z)=\frac{z+a}{az+1}
\end{equation}
for some $0<a<1$ which has two fixed points $1$ and $-1$ on $\mathbb{T}$. Recently, the analogous question for $H^2(\mathbb{C}_+)$ was answered negatively by R. Schroderus and Hans-Olav Tylli \cite[Proposition 3.5]{Schroderus}. They proved that $C_\psi-\lambda I$ is \emph{not} universal for any $\lambda\in\mathbb{C}$ when $\psi$ is a hyperbolic automorphism of $\mathbb{C}_+$. No examples of composition operators with universal translates are known for $H^2(\mathbb{C}_+)$.

The plan of the paper is the following. Section 2 contains some preliminary definitions and results. In Section 3, all linear fractional symbols $\psi$ for which $C_\psi$ has universal translates on $H^2(\mathbb{C}_+)$ are characterized (see Theorem \ref{Main half-plane result}). As a consequence the first such example is discovered.

\begin{thm}If $\psi$ is an affine self-map of $\mathbb{C}_+$, then $C_\psi-\lambda I$ is universal on $H^2(\mathbb{C}_+)$ for some $\lambda\in\mathbb{C}$ if and only if $\psi(w)=aw+b$ with $a\in(1,\infty)$ and $\mathrm{Re}(b)>0$. 
\end{thm}
In Section 4, we characterize all linear fractional composition operators  $C_\phi$ with universal translates on $H^2(\mathbb{D})$ (see Theorem \ref{Hyperbolic comp universal D}). 

\begin{thm} If $\phi$ is a linear fractional self-map of $\mathbb{D}$, then $C_\phi-\lambda I$ is universal on $H^2(\mathbb{D})$ for some $\lambda\in\mathbb{C}$ if and only if $\phi$ has two distinct fixed points outside $\mathbb{D}$.
\end{thm}
The surprising new example here is the \emph{affine} self-map of $\mathbb{D}$  defined by
\[
\phi_a(z)=az+(1-a), \ \ 0<a<1
\]
with fixed points $1$ and $\infty$. It is interesting that R. Schroderus and Hans-Olav Tylli \cite[Example 3.4]{Schroderus} considered $C_{\phi_a}$ and showed that it is not universal (since it is injective), but they did not consider its translates $C_{\phi_a}-\lambda I$. Recently Cowen and Gallardo-Gutiérrez \cite[Theorem 8]{Cowen- Gallardo 4} showed that $C_{\phi_a}$ (with $a=e^{-2\pi}$)  is associated with a particular universal Toeplitz operator on the classical Bergman space. The rest of this work is focused exclusively on analysing $C_{\phi_a}$. 

Section 5 focuses on the minimal invariant subspaces of $C_{\phi_a}$ in $H^2:=H^2(\mathbb{D})$. These subspaces are necessarily \emph{cyclic subspaces}
\[
K_f=\overline{\mathrm{span}\{C^n_{\phi_a} f: n\geq 0\}}^{H^2}
\] 
generated by some $f\in H^2$. Therefore the ISP has a positive solution if and only if $K_f$ is a minimal invariant subspace for $C_{\phi_a}$ precisely when $f$ is an eigenvector. We first show that only those $K_f$ with $f\in H^2$ analytic at each point of $\mathbb{T}\setminus\{1\}$ need to be considered (see Proposition \ref{ISP T\1}). Define the \emph{radial limit} of $f\in H^2$  at $1$ by $f^*(1):=\lim_{r\to1^{-}}f(r)$ if it exists. Therefore resolving the ISP amounts to characterizing the minimality of $K_f$ in the following three cases: \\ 

$\mathbf{A.}$ \emph{$f^*(1)$ is finite and non-zero.} \ \ \ \ \ \ \ \ \ $\mathbf{B.}$ $f^*(1)$ \emph{is equal to zero.} \\

$\mathbf{C.}$ \emph{$f^*(1)$ does not exist.} \\ \\
Let $f_s(z)=(1-z)^s$ for $\mathrm{Re}(s)>-1/2$ which are eigenvectors for all the $C_{\phi_a}$. The main result on minimal invariant subspaces is the following (see Theorem \ref{Main Minimal Invariant Thm}). 
\begin{thm}\label{Main minimal}Let $f=f_sg$ for some $g\in H^2$ and $\mathrm{Re}(s)\geq 0$ with $\lim_{n\to\infty}g(1-a^n)$ finite and non-zero.
	Then $K_f$ is a minimal invariant subspace for $C_{\phi_a}$ if and only if $f$ is a scalar multiple of $f_s$.
\end{thm}
 Putting $s=0$ in Theorem \ref{Main minimal} resolves case $\mathbf{A}$ in full generality for $C_{\phi_a}$. In fact $f^*(1)$ need not even exist. The known results in the literature that treat case $\mathbf{A}$ via the hyperbolic composition operator $C_h$ (see \eqref{Canonical Auto}) all require additional hypotheses (see \cite{Gallardo-Gorkin},\cite{Matache minimal},\cite{Matache eigenfunction},\cite{Mortini}). 

\begin{cor} Let $f\in H^2$ with $\lim_{n\to\infty}f(1-a^n)$ finite and non-zero. Then $K_f$ is minimal for $C_{\phi_a}$ if and only if $f$ is constant.
\end{cor}

 With the additional assumption of analyticity of $f$ at $1$ we also characterize case $\mathbf{B}$ (see Corollary \ref{Minimal anal at 1}). It is worth noting that this result is not true for $C_h$ because the functions $f_N$ are not eigenvectors for $C_h$ (see \cite[Prop. 2.1]{Gallardo-Gorkin} or \cite[Prop. 2.3]{Matache eigenfunction}).

\begin{cor} Let $f\in H^2$ be analytic at $1$. Then $K_f$ is minimal for $C_{\phi_a}$ if and only if $f$ is a scalar multiple of $f_N(z)=(1-z)^N$ for some $N\in\mathbb{N}$.
\end{cor}

We also show that for $f$ belonging to \emph{singular shift-invariant subspaces}, the $K_f$ are always non-minimal (see Proposition \ref{Singular shift}). These functions are examples of case $\mathbf{B}$ not covered by Theorem \ref{Main minimal}. The last section is dedicated to eigenvectors of $C_{\phi_a}$. For instance, if $f$ is an eigenvector that is analytic at the point $1$, then $f$ must be a scalar multiple of $(1-z)^N$ for some $N\in\mathbb{N}$. Similarly if $f^*(1)$ is finite and non-zero, then $f$ must be a constant (see Theorem \ref{Main eigenvector}). In general any eigenvector can be analytically continued to the whole half-plane $\mathbb{H}:=\{z\in\mathbb{C}:\mathrm{Re}(z)<1\}$ (see Theorem \ref{Eigenvector domain}). We finally show that $f_s(z)=(1-z)^s$ for $\mathrm{Re}(s)>-1/2$ are the \emph{only} common eigenvectors in $H^2$ shared by all $C_{\phi_a}$ with $0<a<1$.

\section{Preliminaries}
\subsection{The Hardy space $H^2(\mathbb{D})$} We denote by $\mathbb{D}$ and $\mathbb{T}$ the open unit disk and the unit circle respectively. An analytic function $f$ on $\mathbb{D}$ belongs to the Hardy space $H^2(\mathbb{D})$ if
\[
||f||_{\mathbb{D}}=\sup_{0\leq r<1}\left(\frac{1}{2\pi}\int_0^{2\pi}|f(re^{i\theta})|^2d\theta\right)^{1/2}<\infty.
\]
Similarly let $H^\infty$ denote the space of bounded analytic functions on $\mathbb{D}$. For any $f\in H^2(\mathbb{D})$ and $\zeta\in\mathbb{T}$, the radial limit $f^*(\zeta):=\lim_{r\to 1^-}f(r\zeta)$ exists $m$-a.e. on $\mathbb{T}$, where $m$ denotes the normalized Lebesgue measure on $\mathbb{T}$. If $\phi$ is an analytic self-map of $\mathbb{D}$, the \emph{Nevanlinna counting function} for $\phi$ is defined for $w\in \mathbb{D}\setminus\{\phi(0)\}$ by
\[
N_\phi(w)=\sum_{z\in\phi^{-1}\{w\}}\log\frac{1}{|z|}
\]
where $\phi^{-1}\{w\}$ is the sequence of $\phi$-preimages of $w$ repeated according to their multiplicities. If $w\notin\phi(\mathbb{D})$ then $N_\phi(w)$ is defined to be $0$. We shall need a change of variables formula used by Shapiro in his seminal work on compact composition operators \cite[Corollory 4.4]{Shapiro Annals}:
\begin{equation}\label{Change of variables Shapiro}
||C_\phi f||^2_2=2\int_{\mathbb{D}}|f'(w)|^2 N_\phi(w) dA(w)+|f(\phi(0))|^2
\end{equation}
for any $f$ holomorphic on $\mathbb{D}$ and where $dA$ is the normalized area measure on $\mathbb{D}$.

\subsection{The Hardy space $H^2(\mathbb{C}_+)$} Let $\mathbb{C}_+$ be the open right half-plane. The Hardy space $H^2(\mathbb{C}_+)$ is the Hilbert space of analytic functions on $\mathbb{C}_+$ for which the norm 
\[
||f||_{\mathbb{C}_+}=\left(\sup_{0<x<\infty}\int_{-\infty}^\infty|f(x+iy)|^2dy\right)^{1/2}<\infty.
\]
By the \emph{Paley-Wiener Theorem} the transformation
 defined by
\begin{equation}\label{Paley-Weiner}
(Pf)(w)=\int_{\mathbb{R}_+}f(t)e^{-tw}dt
\end{equation}
is an isometric isomorphism of $L^2(\mathbb{R}_+)$  onto $H^2(\mathbb{C}_+)$, where $\mathbb{R}_+$ denotes the non-negative real numbers. It is known that a composition operator $C_\psi$ on $H^2(\mathbb{C}_+)$ is unitarily equivalent to the \emph{weighted composition operator} $W_\Phi$ on $H^2(\mathbb{D})$ defined by
\begin{equation}\label{equiv C D}
(W_\Phi f)(z)=\frac{1-\Phi(z)}{1-z}f(\Phi(z)) 
\end{equation}
where $\Phi=\gamma^{-1}\circ\psi\circ \gamma:\mathbb{D}\to\mathbb{D}$ and $\gamma(z)=\frac{1+z}{1-z}$ is the \emph{Cayley transform} of $\mathbb{D}$ onto $\mathbb{C}_+$ with $\gamma^{-1}(z)=\frac{z-1}{z+1}$ (see \cite[Lemma 2.1]{Chalendar Partington}).

\subsection{Linear fractional self-maps of $\mathbb{D}$ and $\mathbb{C}_+$}
The linear fractional self-maps 
\[
\phi(z)=\frac{az+b}{cz+d}
\]
of $\mathbb{D}$ with $a,b,c,d\in\mathbb{C}$ satisfying $ad-bc\neq 0$ have two fixed points in $\widehat{\mathbb{C}}=\mathbb{C}\cup\{\infty\}$. If $\phi$ has only one fixed point then it is necessarily on $\mathbb{T}$. Of particular importance to us in relation to universality on $H^2(\mathbb{D})$ are those $\phi$ with two distinct fixed points outside $\mathbb{D}$ which we shall call \emph{hyperbolic maps}. When both fixed points belong to the unit circle $\mathbb{T}$ then it is a \emph{hyperbolic automorphism}. 

Matache \cite{Matache comp halfplane} showed that the only  linear fractional self-maps of $\mathbb{C}_+$ that induce bounded composition operators on $H^2(\mathbb{C}_+)$ are the \emph{affine maps}
\begin{equation}\label{LFT symbol}
\psi(w)=aw+b
\end{equation}
where $a>0$ and Re$(b)\geq 0$. Such a map $\psi$ is said to be of \emph{hyperbolic type} if $a\neq 1$ and is a hyperbolic automorphism if additionally Re$(b)=0$. In particular, the \emph{hyperbolic non-automorphisms} are the symbols 
\[
\psi(w)=aw+b \ \ \mathrm{with}  \  a\in (0,1)\cup(1,\infty) \ \mathrm{and}  \  \mathrm{Re}(b)>0
\]
and we shall say $\psi$ is of \emph{type $\mathrm{I}$} if $a\in(0,1)$ and of \emph{type} $\mathrm{II}$ if $a\in(1,\infty)$. The symbols $\psi$ of \emph{type} $\mathrm{II}$ shall be of interest to us in relation to universality on $H^2(\mathbb{C}_+)$.

\subsection{A necessary condition for universality} R. Schroderus and Hans-Olav Tylli \cite[Corollary 2.3]{Schroderus} proved the following necessary condition for the universality of a bounded operator $T$ on a separable Hilbert space $\mathcal{H}$.

\begin{prop}\label{Non-universality Schroderus} If $T$ is universal on $\mathcal{H}$, then the point spectrum $\sigma_p(T)$ of $T$ has non-empty interior in $\mathbb{C}$.
\end{prop}

Using this we see that most linear fractional composition operators on $H^2(\mathbb{D})$ and $H^2(\mathbb{C}_+)$ do \emph{not} have universal translates.

\begin{prop}\label{Non-universality} Let $\phi$ and $\psi$ be linear fractional self-maps of $\mathbb{D}$ and $\mathbb{C}_+$ respectively. If $\phi$ is not a hyperbolic map, then $C_\phi-\lambda$ is not universal on $H^2(\mathbb{D})$ for any $\lambda\in\mathbb{C}$. Similarly if $\psi$ is not a hyperbolic non-automorphism, then $C_\psi-\lambda$ is not universal on $H^2(\mathbb{C}_+)$ for any $\lambda\in\mathbb{C}$.
	\end{prop}
\begin{proof} We first consider the disk case. Suppose $\phi(\alpha)=\alpha$ for some $a\in\mathbb{D}$. Then the classical Königs Theorem shows that $\sigma_p(C_\phi)\subset\{1\}\cup(\lambda^n)_{n\in\mathbb{N}}$ where $\lambda=\phi'(a)$  \cite[Theorem 5.3.3]{Avendano Rosenthal} and hence $\sigma_p(C_\phi)$ has empty interior. Now suppose $\phi(\alpha)=\alpha$ for some $\alpha\in\mathbb{T}$ and $\alpha$ is the only fixed point in $\widehat{\mathbb{C}}$. If $\phi$ is an automorphism of $\mathbb{D}$ then $\sigma_p(C_\phi)=\mathbb{T}$ \cite[Theorem 5.4.6]{Avendano Rosenthal}, and if $\phi$ is a non-automorphism then $\sigma_p(C_\phi)$ is contained in a \emph{spiral-like} set $\{0\}\cup\{e^{-at}:t\in[0,\infty)]\}$ for some $a\in\mathbb{C}$ with $\mathrm{Re}(a)>0$ (see \cite{Cowen}). In both cases $\sigma_p(C_\phi)$ has empty interior. Since $\sigma_p(C_\phi-\lambda)$ is just a translate of $\sigma_p(C_\phi)$ for each $\lambda\in\mathbb{C}$, it also has empty interior. Therefore $C_\phi-\lambda$ is not universal for any $\lambda\in\mathbb{C}$ if $\phi$ is not hyperbolic.
	
	For the half-plane case suppose $\psi(w)=aw+b$ with $a=1$ or with $a\neq 1$ and $\mathrm{Re}(b)=0$. That is precisely when $\psi$ is \emph{not} a hyperbolic non-automorphism. Then $\sigma_p(C_\psi)$ is contained in a circle or a spiral (see \cite[Theorem 7.4]{Eva Adjoints Dirichlet}) and hence has empty interior. Therefore once again $C_\psi-\lambda$ is not universal for any $\lambda\in\mathbb{C}$ if $\phi$ is not a hyperbolic non-automorphism. 
	\end{proof}

In the next section we shall prove that $C_\psi-\lambda$ is universal  on $H^2(\mathbb{C}_+)$ when $\psi$ is a hyperbolic non-automorphism of type $\mathrm{II}$ with $\lambda$ in the interior of $\sigma_p(C_\psi)$. This is to our knowledge the first example of a composition operator with universal translates on $H^2(\mathbb{C}_+)$. 
\section{Universality on $H^2(\mathbb{C}_+)$} In this section we completely characterize the affine self-maps $\psi$ of $\mathbb{C}_+$ for which $C_\psi-\lambda$ is universal on $H^2(\mathbb{C}_+)$ for some $\lambda\in\mathbb{C}$. We shall see that this happens only for hyperbolic non-automorphism of type $\mathrm{II}$. We must first show that $C_\psi$ is unitarily equivalent to a weighted vector shift when $\psi$ is a hyperbolic non-automorphism. Partington and Pozzi \cite{Partington-Pozzi} have shown that a wide class of vector shift operators are universal in the sense of Caradus. Let $T:\ell^2(\mathbb{Z},L^2(t_0,t_1))\to\ell^2(\mathbb{Z},L^2(t_0,t_1))$
be the weighted \emph{right} bilateral shift given by 
\[
T(\sum_{n\in\mathbb{Z}}x_ne_n)=\sum_{n\in\mathbb{Z}}k_{n}x_{n-1}e_n
\]
where each $k_n$ is a positive continuous function on $[t_0,t_1]$ such that 

\begin{equation*}
k_n \xrightarrow{\text{uniformly}}
\begin{cases}
b & \mathrm{as} \ \ n\to -\infty \\
a & \mathrm{as} \ \ n\to +\infty \\
\end{cases}  
\end{equation*}
where $a<b$.
 We state their main result as follows. 
\begin{thm}\label{Partington-Pozzi}
 For any $\lambda\in\mathbb{C}$ with $a<|\lambda|<b$, the operator $T-\lambda I$ is a universal operator on $\ell^2(\mathbb{Z},L^2(t_0,t_1))$. Also $\sigma(T)=\sigma(T^*)=\{z\in\mathbb{C}:a\leq|z|\leq b\}$ with $\{z\in\mathbb{C}:a<|z|< b\}\subset \sigma_p(T)$ and $\sigma_p(T^*)=\emptyset$.
\end{thm}

We begin by showing that non-automorphic $C_\psi$ are unitarily equivalent to a dilation followed by a multiplication operator on $L^2(\mathbb{R}_+)$. The case $a=1$ already appears in \cite[Thm. 7.1]{Eva Adjoints Dirichlet}.
\begin{lem}\label{equivalence to a weighted comp} Let $a\in(0,\infty)$, $b\in\mathbb{C_+}$ and $\psi:\mathbb{C}_+\to \mathbb{C}_+$ be the symbol $\psi(w)=a w + b$. Then the composition operator $C_{\psi}:H^2(\mathbb{C}_+)\rightarrow H^2(\mathbb{C}_+)$ is unitarily equivalent to the operator $W:L^2(\mathbb{R}_+)\rightarrow L^2(\mathbb{R}_+)$ defined by $(Wf)(t)=\frac{1}{a}e^{-bt/a}f(t/a)$.
	\end{lem}
\begin{proof} By the Paley-Wiener theorem, the map 
	$P : L^2(\mathbb{R}_+)\rightarrow H^2(\mathbb{C}_+)$ defined by
\[
(Pf)(w)=\int_{\mathbb{R}_+}f(t)e^{-tw}dt
\]
is an isometric isomorphism.
Let $f \in L^2(\mathbb{R}_+)$ and $F:=P(f)\in H^2(\mathbb{C}_+)$. We get 
\begin{align*}
(C_{\psi} Pf)(w)&=C_{\psi}F(w)=F(aw +b)=\int_{\mathbb{R}_+}f(t)e^{-t(a w + b)}dt=\int_{\mathbb{R}_+}f(t)e^{-bt}e^{-atw}dt \\
&=\int_{\mathbb{R}_+}\frac{1}{a}e^{-bt/a}f(t/a)e^{-tw}dt=(PWf)(w).
\end{align*}
Hence $C_\psi$ on $H^2(\mathbb{C}_+)$ is unitarily equivalent to $W$ on $L^2(\mathbb{R}_+)$.
\end{proof}

The next result shows that the operator $W$ with $a\neq 1$ is unitarily equivalent to a weighted vector \emph{left} shift.

\begin{lem}\label{Equivalence to shift} For $a \in (0,1) \cup (1,+\infty)$, the operator  $W:L^2(\mathbb{R}_+)\rightarrow L^2(\mathbb{R}_+)$ defined by $(Wf)(t)=\frac{1}{a}e^{-bt/a}f(t/a)$ is unitarily equivalent to the weighted left bilateral shift $T$ defined by  
	\[
	T(\sum_{n\in\mathbb{Z}}g_ne_n)=\sum_{n\in\mathbb{Z}}c_{n}g_{n+1}e_n
	\]
	 on $l^2(\mathbb{Z}, L^2[1,a])$ for $a>1$ (resp. $l^2(\mathbb{Z},  L^2[a, 1])$  for $a<1$), and where  \[c_n(t):=a^{-1/2}e^{-Re(b)a^{-n-1}t}\]
	 are positive and continuous functions on $[1,a]$ (resp. $[a,1]$).
\end{lem}
\begin{proof} Let $f\in L^2(\mathbb{R}_+)$ and consider first the case $a>1$. Then after a change of variables one obtains
\begin{align}\label{isometry equation}
 ||f||^2_{L^2(\mathbb{R}_+)}&=\int_{\mathbb{R}_+}|f(t)|^2dt
=\sum_{n \in \mathbb{Z}}\int_{a^{-n}}^{a^{-n+1}}|f(t)|^2dt
=\sum_{n \in \mathbb{Z}}\int_{1}^{a}a^{-n}|f(t/a^{n})|^2dt \nonumber\\
&=\sum_{n \in\mathbb{Z}}\int_{1}^{a}|a^{-n/2}f(t/a^{n})|^2dt 
\end{align}
Define a sequence of unimodular functions by
\begin{equation*}
a_n(t): =
\begin{cases}
\prod_{k=0}^n e^{-itIm(b)a^{-k}} & \mathrm{if} \ \ n\geq 0 \\
e^{-itIm(b)}\prod_{k=0}^{-n-1}e^{itIm(b)a^{k}} & \mathrm{if} \ \ n<0
\end{cases}       
\end{equation*}
and note that $(a_n/a_{n+1})(t)=e^{itIm(b)a^{-n-1}}$ for all $n\in\mathbb{Z}$. Define the operator $\Psi:L^2(\mathbb{R}_+)\rightarrow l^2(\mathbb{Z}, L^2[1,a])$ by 
\[
\Psi (f)=\sum_{n \in \mathbb{Z}}h_ne_n
\]
 where $h_n(t)=a^{-n/2} a_n(t) f(t/a^{n})$ for $t\in[1,a]$ and $e_n$ a canonical basis vector. Then \eqref{isometry equation} and $|a_n(t)|=1$ show that
 \[
 ||\Psi(f)||_{\ell^2}^2=\sum_{n\in\mathbb{Z}}||h_n||_{L^2[1,a]}^2=||f||^2_{L^2(\mathbb{R}_+)}
 \]
for each $f\in L^2(\mathbb{R_+})$ and so $\Psi$ is an isometry. For surjectivity let $H=\sum_{n \in \mathbb{Z}}h_ne_n$ with $\sum_{n\in\mathbb{Z}}||h_n||_{L^2[1,a]}^2<\infty$. If we define $f\in L^2(\mathbb{R}_+)$ by $f(t)=a^{n/2}h_n(a^nt)/a_n(a^nt)$ for $t\in[a^{-n},a^{-n+1}]$ noting that $\mathbb{R}_+=\bigcup_{n\in\mathbb{Z}}[a^{-n},a^{-n+1}]$, then $\Psi(f)=H$ and therefore $\Psi$ is unitary. Hence for $f\in L^2(\mathbb{R_+})$, we get
\begin{align*}
	(\Psi\circ W)f&= \Psi\left(a^{-1}e^{-bt/a}f(t/a)\right) = \sum_{n\in\mathbb{Z}}a^{-1}e^{-bta^{-n-1}}a^{-n/2}a_nf(t/a^{n+1})e_n\\
	&=\sum_{n\in\mathbb{Z}}a^{-1/2}e^{-bta^{-n-1}}\frac{a_n}{a_{n+1}}h_{n+1}e_n\\
	&=\sum_{n\in\mathbb{Z}}a^{-1/2}e^{-bta^{-n-1}}e^{itIm(b)a^{-n-1}}h_{n+1}e_n\\
	&=\sum_{n\in\mathbb{Z}}a^{-1/2}e^{-Re(b)ta^{-n-1}}h_{n+1}e_n\\
	&=\sum_{n\in\mathbb{Z}}c_nh_{n+1}e_n=T(\sum_{n\in\mathbb{Z}}h_{n}e_n)=(T\circ \Psi) f.
\end{align*}
Therefore $W$ is unitarily equivalent to $T$ when $a>1$. The case $0<a<1$ is analogous with the only changes being to replace $l^2(\mathbb{Z}, L^2[1,a])$ by $l^2(\mathbb{Z}, L^2[a,1])$ and equation \eqref{isometry equation} which becomes
\begin{align*}
||f||^2_{L^2(\mathbb{R}_+)}&
=\sum_{n \in \mathbb{Z}}\int_{a^{n+1}}^{a^{n}}|f(t)|^2dt
=\sum_{n \in \mathbb{Z}}\int_{a}^1a^{n}|f(ta^{n})|^2dt 
\\
&=\sum_{n \in\mathbb{Z}}\int_{a}^1|a^{n/2}f(ta^{n})|^2dt=\sum_{n\in\mathbb{Z}}||h_n||_{L^2[a,1]}^2 .
\end{align*}
The functions $a_n$, $h_n$ and operator $\Psi:L^2(\mathbb{R}_+)\rightarrow l^2(\mathbb{Z}, L^2[a,1])$ are defined just as in the previous case and the rest of the proof follows verbatim.
\end{proof}

We therefore arrive at the main result. As a by-product we get the spectra and point spectra of hyperbolic non-automorphic composition operators which have been obtained recently by Schroderus \cite{Schroderus spectra}.

\begin{thm}\label{Universal Comp halfplane} Let $a\in(0,1)\cup(1,\infty)$, $b\in\mathbb{C_+}$ and $\psi$ a hyperbolic non-automorphism defined on $\mathbb{C}_+$ by $\psi(w)=aw + b$. For $a>1$, the operator $C_\psi-\lambda$ is universal on $H^2(\mathbb{C}_+)$ for $0<|\lambda|<a^{-1/2}$. If $a<1$ then  $C_\psi-\lambda $ is not universal for any $\lambda\in\mathbb{C}$. Furthermore $\{\lambda\in\mathbb{C}:0<|\lambda|<a^{-1/2}\}\subset\sigma_p(C_\psi)$ if $a>1$ and $\sigma_p(C_\psi)=\emptyset$ if $a<1$. In both cases $\sigma(C_\psi)=\{\lambda\in\mathbb{C}:|\lambda|\leq a ^{-1/2}\}$.
\end{thm}
\begin{proof} By Lemma \ref{equivalence to a weighted comp} and Lemma \ref{Equivalence to shift} we see that $C_\psi$ is unitarily equivalent to the weighted left bilateral shift with weights
\[
c_n(t)=a^{-1/2}e^{-Re(b)a^{-n-1}t}.
\]
In order to apply Theorem \ref{Partington-Pozzi} we must first transform these \emph{left} shifts into \emph{right} shifts. We observe that any such left shift is unitarily equivalent to the corresponding right shift with reversed weights 
\[
\widetilde{c_n}(t):=c_{-n}(t)=a^{-1/2}e^{-Re(b)a^{n-1}t}
\]
and its adjoint is equivalent to the right shift with the original weights $(c_n)_{n\in\mathbb{N}}$. Hence when $a>1$, we see that $\widetilde{c_n}\To 0$ as $n\to \infty$
 and $\widetilde{c_n}\To a^{-1/2} $ as $n\to -\infty$	uniformly on
 $[1,a]$. So $C_\psi-\lambda $ is universal for $0<|\lambda|<a^{-1/2}$ by Theorem \ref{Partington-Pozzi}. For the case $a<1$, we have $c_n\To 0$ as $n\to \infty$ and $c_n\To a^{-1/2}$ as $n\to -\infty$ uniformly on $[a,1]$ which implies that $C_\psi^*-\lambda$ is universal for $0<|\lambda|<a^{-1/2}$ and $\sigma_p(C_\psi)=\sigma_p((C^{*}_\psi)^*)=\emptyset$. Therefore $C_\psi-\lambda $ is not universal for any $\lambda\in\mathbb{C}$ by Proposition \ref{Non-universality Schroderus}. The statements on the spectrum and point spectrum follow from Theorem \ref{Partington-Pozzi}.
\end{proof}

Therefore by Theorem \ref{Universal Comp halfplane} and Proposition \ref{Non-universality} we obtain our desired characterization of universality on $H^2(\mathbb{C}_+).$

\begin{thm}\label{Main half-plane result} If $\psi$ is an affine self-map of $\mathbb{C}_+$, then $C_\psi-\lambda$ is universal on $H^2(\mathbb{C}_+)$ for some $\lambda\in\mathbb{C}$ if and only if $\psi$ is a hyperbolic non-automorphism of type $\mathrm{II}$. 
\end{thm}

\section{Universality on $H^2(\mathbb{D})$}
In this section we completely characterize the linear fractional self-maps $\phi$ of $\mathbb{D}$ for which $C_\phi-\lambda$ is universal on $H^2(\mathbb{D})$ for some $\lambda\in\mathbb{C}$. Recall that $\phi$ is hyperbolic if it has two distinct fixed points outside $\mathbb{D}$ and a hyperbolic automorphism when both fixed points belong to the unit circle $\mathbb{T}$. At least one fixed point must necessarily belong to $\mathbb{T}$. The main result is the following.
\begin{thm}\label{Hyperbolic comp universal D}
If $\phi$ is a linear fractional self-map of $\mathbb{D}$, then $C_\phi-\lambda$ is universal on $H^2(\mathbb{D})$ for some $\lambda\in\mathbb{C}$ if and only if $\phi$ is hyperbolic.
\end{thm}

 This extends the  well-known result of Nordgren, Rosenthal and Wintrobe \cite[Theorem 6.2]{Nordgren-Rosenthal-Wintrobe} where they proved this for hyperbolic automorphisms. 
 
 So let $\phi$ be a hyperbolic non-automorphism.  Hence $\phi$ fixes one point $\zeta\in\mathbb{T}$ and the other outside the closed unit disk $\overline{\mathbb{D}}$ (possibly $\infty$). It was shown by Hurst \cite[Theorem 8]{hust} that in this case $C_\phi$ is similar to $C_{\phi_a}$ where 
 \begin{equation}\label{normalized hyperbolic nonauto}
 \phi_a(z)=az+(1-a) \ \ \mathrm{with} \ \ a:=\phi'(\zeta)\in(0,1).
 \end{equation}
The operator $C_{\phi_a}$ was studied by Deddens (see \cite{Deddens}) where it was shown that the adjoint of $C_{\phi_a}$ is subnormal and where its spectrum was determined. Interestingly $C_{\phi_a}$ is unitarily equivalent to a scalar multiple of a composition operator on $H^2(\mathbb{C}_+)$ with symbol that is a hyperbolic non-automorphism of type $\mathrm{II}$.
\begin{lem} \label{equiv between C and D}Let $a\in (0,1)$ and $\phi_a$ be the self map of $\mathbb{D}$ given by $\phi_a(z)=a z+(1-a)$. Then $C_{\phi_a}$ on $H^2(\mathbb{D})$ is unitarily equivalent to $a^{-1}C_{\psi_a}$ on $H^2(\mathbb{C}_+)$ where \[\psi_a(w)=a^{-1}w+(a^{-1}-1).\]
\end{lem}
\begin{proof} We only have to determine \eqref{equiv C D} with $\Phi=\gamma^{-1}\circ\psi_a\circ \gamma$. We get
	\begin{align*}\label{620}
\Phi(z)&= \gamma^{-1}\left(a^{-1}\gamma(z)+(a^{-1}-1)\right)\nonumber
	= \displaystyle \frac{a^{-1}\gamma(z)+(a^{-1}-1)-1}{a^{-1}\gamma(z)+(a^{-1}-1)+1}\nonumber\\
	&= \displaystyle \frac{a^{-1}\left(\frac{1+z}{1-z}\right)+a^{-1}-2}{a^{-1}\left(\frac{1+z}{1-z}\right)+a^{-1}}= \frac{a^{-1}(1+z)+(a^{-1}-2)(1-z)}{a^{-1}(1+z)+a^{-1}(1-z)}\nonumber\\
	&=\frac{2a^{-1}-2+2z}{2a^{-1}}=a z+(1-a)=\phi_a(z).
	\end{align*}
Similarly \[\frac{1-\Phi(z)}{1-z}=\frac{1-\phi_a(z)}{1-z}=\frac{a(1-z)}{1-z}=a.\] 
Therefore we see that $C_{\psi_a}$ is unitarily equivalent to $W_\Phi=aC_{\phi_a}$ by \eqref{equiv C D}. 
\end{proof}

We are now ready to prove universality in the hyperbolic non-automorphism case, which together with the hyperbolic automorphism case (see \cite[Thm. 6.2]{Nordgren-Rosenthal-Wintrobe}) and Proposition \ref{Non-universality} proves Theorem \ref{Hyperbolic comp universal D}.

\begin{thm}\label{hyperbolic nonaut univ D}
Let $\phi$ be a hyperbolic non-automorphism with one fixed point $\zeta\in\mathbb{T}$ and the other outside the closed unit disk $\overline{\mathbb{D}}$ (possibly at $\infty$). If $a:=\phi'(\zeta)\in(0,1)$, then for each $\lambda$ with $0<|\lambda|<a^{-1/2}$ the operator $C_\phi-\lambda$ is universal on $H^2(\mathbb{D})$. 
\end{thm}
\begin{proof} By Lemma \ref{equiv between C and D} and the paragraph before it we see that $C_\phi$ on $H^2(\mathbb{D})$ is similar to $a^{-1}C_{\psi_a}$ on $H^2(\mathbb{C}_+)$. Since $\psi_a(s)=a^{-1}s+(a^{-1}-1)$ with $a^{-1}>1$, it follows that $\psi_a$ is a hyperbolic non-automorphism of type $\mathrm{II}$. Hence $C_{\psi_a}-\lambda$ is universal for $0<|\lambda|<a^{1/2}$ by Theorem \ref{Universal Comp halfplane}, and therefore $C_\phi-\lambda $ must be universal for $0<|\lambda|<a^{-1/2}$. 
\end{proof}

Since $C_\phi$ and $C_\phi-\lambda$ have the same invariant subspaces, we get the following.

\begin{cor} \label{Min inv 1-dim} Let  $\phi:\mathbb{D}\to\mathbb{D}$ be a hyperbolic map. Then the ISP has a positive solution if and only if every  minimal non-trivial invariant subspace of $C_\phi$ in $H^2(\mathbb{D})$ is one dimensional.
	\end{cor}

In the rest of this work we shall exclusively focus on the \emph{canonical} hyperbolic  non-automorphism defined on $\mathbb{D}$ by \[\phi_a(z)=az+(1-a), \ \ 0<a<1\]
with $1$ and $\infty$ as the fixed points outside $\mathbb{D}$. By Corollory \ref{Min inv 1-dim} it is clear that a deeper understanding of the minimal invariant subspaces of $C_{\phi_a}$ is key to attacking the ISP. Therefore the next section is dedicated to the minimal invariant subspaces of $C_{\phi_a}$ on $H^2(\mathbb{D})$. 

\section{Minimal invariant subspaces of $C_{\phi_a}$} 

 In the rest of this work we shall denote $H^2:=H^2(\mathbb{D})$. In this section we study the minimal invariant subspaces of $C_{\phi_a}$ for $a\in(0,1)$. First note that for $a,b\in(0,1)$ we get
 \[
 \phi_a\circ\phi_b(z)=abz+1-ab=\phi_{ab}(z)
 \]
 and hence that $C_{\phi_a}C_{\phi_b}=C_{\phi_{ab}}=C_{\phi_b}C_{\phi_a}$. It follows that $\{C_{\phi_a}:a\in(0,1)\}$ is a multiplicative semigroup of operators. In particular the $n$-th compositional iterate of $\phi_a$ is simply given by
 \begin{equation}\label{Compositional iterates}
 \phi_a^{[n]}=a^nz+(1-a^n)=\phi_{a^n}(z)
 \end{equation}
 for each $n\in\mathbb{N}$ and $C^n_{\phi_a}=C_{\phi_a^{[n]}}=C_{\phi_{a^n}}$.  Hence $\phi_a^{[n]}\To 1$ uniformly on $\overline{\mathbb{D}}$ as $n\to\infty$. For each non-zero $f\in H^2$, we denote by $K_f$ the \emph{cyclic subspace} defined by
\[
K_f=\overline{\mathrm{span}\{C^n_{\phi_a} f: n\geq 0\}}^{H^2}.
\] 
Each such $K_f$ is an invariant subspace for $C_{\phi_a}$. If $K_f=H^2$ then $f$ is called a \emph{cyclic vector} for $C_{\phi_a}$. If $E$ is an invariant subspace of $C_{\phi_a}$, then $K_f\subset E$ for all $f\in E$. So if $E$ is a minimal $C_{\phi_a}$-invariant subspace then $K_f=E$ for all $f\in E$. Hence minimal invariant subspaces are \emph{necessarily} cyclic. Noting that $\mathrm{dim}(K_f)=1$ precisely when $f$ is a $C_{\phi_a}$-eigenvector, we may restate the ISP as follows. The ISP has a positive solution if and only if $K_f$ is minimal precisely when $f$ is a $C_{\phi_a}$-eigenvector. Our first result shows that we need only consider $f\in H^2$ that are analytic on the unit circle $\mathbb{T}$ minus the point $1$.

\begin{prop}\label{ISP T\1} Each minimal invariant $K_f$ has a generator  that is analytic on a neighborhood of $\overline{\mathbb{D}}\setminus\{1\}$. This neighborhood can be chosen large enough to include any given compact subset of the half-plane $\mathbb{H}=\{z\in\mathbb{C}:\mathrm{Re}(z)<1\}$.
\end{prop}
\begin{proof} Consider the sequence of open disks $\mathbb{D}_n$ with center $1-a^{-n}$ and radius $a^{-n}$. Then $\mathbb{D}=\mathbb{D}_0\subset\mathbb{D}_1\subset\ldots$ is an increasing chain of disks having $1$ as the common boundary point and centers $1-a^{-n}$ tending to $-\infty$. Therefore $\mathbb{H}=\cup_{n=0}^\infty\mathbb{D}_n$. Also we have 
	$\phi_a(\mathbb{D}_n)=a\mathbb{D}_n+(1-a)=\mathbb{D}_{n-1}$ for $n\geq 1$ which implies $\phi_{a^n}(\mathbb{D}_n)=\mathbb{D}$. Clearly $\overline{\mathbb{D}}\setminus\{1\}\subset\mathbb{D}_n$ for all $n\geq 1$. Then for a given compact subset $L$ of $\mathbb{H}$ there exists $N\in\mathbb{N}$ such that $L\subset\mathbb{D}_N$.  For any $f\in H^2$ define the function $f_N:=f\circ\phi_{a^N}\in H^2$ which is analytic on $\mathbb{D}_N$ and $K_{f_N}\subset K_f$. The key observation is that if $K_f$ is minimal then $K_{f_N}=K_f$. So each minimal $K_f$ has a generator analytic on $\mathbb{D}_N$. 
\end{proof}

The principle examples of eigenvectors for $C_{\phi_a}$ are the functions 
\[
f_s(z)=(1-z)^s
\]
for $s\in\mathbb{C}$. In fact for $z\in\mathbb{D}$ we have
\begin{equation}\label{Principal eigenvector}
(C_{\phi_a}f_s)(z)=(1-(az+1-a))^s=a^s(1-z)^s=a^sf_s(z).
\end{equation}
So $C_{\phi_a}f_s=a^sf_s$ for all $a\in(0,1)$ and Hurst \cite[Lemma 7]{hust} showed that $f_s\in H^2$ if and only if $\mathrm{Re}(s)>-1/2$. We now state our main theorem on minimal invariant subspaces.
\begin{thm}\label{Main Minimal Invariant Thm} Let $f=f_sg$ for some $g\in H^2$ with $\lim_{n\to\infty} g(1-a^n)=L\neq 0$ and $\mathrm{Re}(s)\geq 0$.
	Then $K_f$ contains the eigenvector $f_s$. Therefore $K_f$ is minimal invariant if and only if $f=Lf_s$.
\end{thm}
\begin{proof} By the hypothesis $f_s$ is an $H^2$-multiplier for $\mathrm{Re}(s)\geq 0$ we have $f_sg\in H^2$. Since $C^n_{\phi_a}f=a^{ns}f_sC^n_{\phi_a}g$ and $\phi_{a^{n}}(0)=1-a^n$, by \eqref{Change of variables Shapiro} we get
	\begin{align*}
	\norm{\frac{C^n_{\phi_a} f}{a^{ns}}-Lf_s}^2_2&=\norm{f_s(C_{\phi_{a}}^ng-L)}^2_2 
	\leq M\norm{C_{\phi_{a}}^n(g-L)}^2_2= M\norm{C_{\phi_{a^n}}(g-L)}^2_2\\
	&=2M\int_{\mathbb{D}}|g'(w)|^2 N_{\phi_{a^n}}(w) dA(w)+|g(1-a^n)-L|^2.
	\end{align*}
	We need only prove that the last integral tends to $0$ as $n\to\infty$. This will follow by a monotone convergence argument. Notice that the images $\phi_{a^n}(\mathbb{D})=a^n\mathbb{D}+(1-a^n)$ are open disks of decreasing radii $a^n$ with centers $1-a^n$ tending to $1$. Therefore for each $w\in\mathbb{D}$ we have $w\notin\phi_{a^n}(\mathbb{D})$ and hence $N_{\phi_{a^n}}(w)=0$ for all sufficiently large $n$. So $N_{\phi_{a^n}}$ is a monotonically decreasing positive function on $\mathbb{D}$ with pointwise limit $0$. Hence the integral above vanishes as claimed and with $g(1-a^n)\to L$ as $n\to\infty$ this implies that $C^n_{\phi_a} f/a^{ns}\to Lf_s$ in $H^2$. Therefore the eigenvector $f_s\in K_f$ and $K_f$ is minimal if and only if $K_f=\mathbb{C}f_s$, in which case $f=Lf_s$. 
\end{proof}

 The case $s=0$ provides a complete characterization of the minimal $C_{\phi_a}$-invariant $K_f$ when $f^*(1)$ is finite and non-zero. In fact $f^*(1)$ need not even exist.

\begin{cor}\label{Minimal f* not 0}	Let $f\in H^2$ with $\lim_{n\to\infty}f(1-a^n)=L\neq0$. Then $K_f$ is a minimal invariant subspace for $C_{\phi_a}$ if and only if $f$ is the constant $L$.
\end{cor}

Assuming analyticity of $f$ at $1$, we can also completely characterize the minimal invariant $K_f$. 

\begin{cor} \label{Minimal anal at 1}Let $f\in H^2$ be analytic at $1$. Then $K_f$ is minimal if and only if $f$ is a scalar multiple of $f_N(z)=(1-z)^N$ for some $N\in\mathbb{N}$.
\end{cor}
\begin{proof} If $f(1)\neq 0$ then the result follows from the previous corollary. Therefore let $f(1)=0$. Hence there exists a neighborhood $U$ of $1$ such that $f=f_Ng$ for a function $g$ analytic on $U$ with $g(1)=L\neq 0$ and $N$ the multiplicity of $f$ at $1$. Now \[C_{\phi_{a^n}}f=a^{nN}f_N g\circ \phi_{a^n}\]
	and if $n$ is sufficiently large say $n>k$ then $\phi_{a^n}(\mathbb{D})=a^n\mathbb{D}+(1-a^n)\subset U$. Hence $g\circ \phi_{a^n}$ is a bounded holomorphic function on $\mathbb{D}$ for $n>k$ with $g\circ \phi_{a^n}(1)=L$. Now applying Theorem \ref{Main Minimal Invariant Thm} with $h:=C_{\phi_{a^n}}f/a^{nN}$ for some $n>k$ and $s=N$ implies that the eigenvector $f_N\in K_h\subset K_f$ which concludes the proof.
\end{proof}

For $b>0$, let $E_b$ denote the \emph{singular shift-invariant} subspace 
\[
E_b=e^{b\frac{z+1}{z-1}}H^2(\mathbb{D}).
\]
It is clear that $E_b\subset E_{b'}$ if $b'<b$ because $e^{b'\frac{z+1}{z-1}}$ divides $e^{b\frac{z+1}{z-1}}$ as an inner function. Cowen and Wahl (see \cite[Theorem 5]{Cowen-Rebecca}) showed that if $\phi$ is any self-map of the disk with $\phi(1)=1$ and $\phi'(1)\leq 1$, then each $E_b$ is an invariant subspace for $C_\phi$. In particular $C_{\phi_a}E_b\subset E_b$ for all $b>0$. None of these $f\in E_b$ satisfy the hypothesis of Theorem \ref{Main Minimal Invariant Thm} since $(f/f_s)^*(1)=0$ for all $f\in E_b$ and $s\in\mathbb{C}$. However the corresponding $K_f$ are not minimal.

\begin{thm}\label{Singular shift} For any $b>0$ and $f\in E_b$ non-zero, the cyclic subspace $K_f$ is not minimal invariant for $C_{\phi_a}$. 
\end{thm}
\begin{proof}
	First note that
	\[
	C_{\phi_a}e^{b\frac{z+1}{z-1}}=e^{\frac{b}{a}\frac{az-a+2}{z-1}}=e^{\frac{b}{a}\frac{z+1}{z-1}}e^{\frac{b}{a}(a-1)}\in E_{b/a}
	\]
	which implies that $C^n_{\phi_a}E_b\subset E_{b/a^n}$ for all $n\geq 1$ and clearly $E_{b/a^n}\subset E_b$. We next prove that for each non-zero $f\in H^2(\mathbb{D})$ there exists an integer $N$ large enough (depending on $f$) such that $f\notin E_N$. Otherwise the inner part of $f$ would be divisible by each of the singular inner functions $I_n(z):=e^{n\frac{z+1}{z-1}}$ for $n\in\mathbb{N}$. But this implies that 
	\[
	2\pi n=\mu_n(\{1\})\leq\mu_f(\{1\})
	\]
	for all $n\in\mathbb{N}$, where $\mu_n$ and $\mu_f$ are the singular measures on $\mathbb{T}$ corresponding to $I_n$ and $f$ respectively (see \cite[Theorem 2.6.7]{Avendano Rosenthal}). Hence $\mu_f(\{1\})=\infty$ which is a contradiction. Now let $f\in E_b\setminus\{0\}$ for some $b>0$ and suppose $f\notin E_N$ for some $N>b$. Then there exists $n_0$ such that $C^n_{\phi_a}f\in E_{b/a^n}\subset E_N$ for all $n\geq n_0$. So if $g:=C^{n_0}_{\phi_a}f$, then $K_g\subset E_N$ which implies that $f\notin K_g$. Therefore $K_g$ is a proper closed invariant subspace of $K_f$ under $C_{\phi_a}$ and hence $K_f$ is not minimal.
\end{proof}

\section{Eigenvectors of $C_{\phi_a}$}
In this final section we study the properties of eigenvectors of $C_{\phi_a}$ for $a\in(0,1)$. The spectrum $\sigma(C_{\phi_a})=\{\lambda\in\mathbb{C}:|\lambda|\leq a ^{-1/2}\}$ was proved in \cite[Thm. 3.1]{Kamowitz} but in greater generality. It follows by Theorem \ref{Universal Comp halfplane} and Lemma \ref{equiv between C and D} that the point spectrum of $C_{\phi_a}$ contains $\{\lambda\in\mathbb{C}:0<|\lambda|<a^{-1/2}\}$. Our first result about eigenvectors follows immediately from Corollary \ref{Minimal f* not 0} and Corollary \ref{Minimal anal at 1}.
 
 \begin{thm}\label{Main eigenvector} Let $f$ be a non-zero $C_{\phi_a}$-eigenvector. Then $f^*(1)$ is finite and nonzero if and only if $f$ is a constant. And $f$ is analytic at the point $1$ if and only if $f(z)=K(1-z)^n$ for some $n\in\mathbb{N}$ and scalar $K$.
 \end{thm}

We turn now to concrete function-theoretic properties of eigenvectors. We first show that the domain of analyticity of eigenvectors for $C_{\phi_a}$ can be extended to the half-plane $\mathbb{H}=\{z\in\mathbb{C}:\mathrm{Re}(z)<1\}$.
\begin{prop}\label{Eigenvector domain} If $f:\mathbb{D}\to\mathbb{C}$ is an eigenvector for $C_{\phi_a}$ for some $a\in(0,1)$, then $f$ has an analytic continuation to all of $\mathbb{H}$. 
		\end{prop}
\begin{proof} Let $\mathbb{D}_n$ be the increasing sequence of disks with center $1-a^{-n}$, radius $a^{-n}$ and $1$ as their only common boundary point. Then $\mathbb{H}=\cup_{n=0}^\infty\mathbb{D}_n$ and $\phi_{a^n}(\mathbb{D}_n)=\mathbb{D}$ as in the proof of Proposition \ref{ISP T\1}. Therefore if $C_{\phi_a}f=\lambda f$ then $\lambda\neq 0$ and
	\[
	f=\frac{C_{\phi_a}^n f}{\lambda^n}=\frac{C_{\phi_{a^n}}f}{\lambda^n}=\frac{f\circ\phi_{a^n}}{\lambda^n}.
	\]
	This clearly implies that $f$ can be analytically continued to $\mathbb{D}_n$ for each $n\in\mathbb{N}$ and hence to all of $\mathbb{H}$.
		\end{proof}
	We next show that if an eigenvector has a zero in $\mathbb{H}$ then it has infinitely many.
	 \begin{prop}\label{Zeros of eigenvectors} If $f$ is a non-zero $C_{\phi_a}$-eigenvector and $f(w)=0$ for some $w\in\mathbb{H}$, then $f(a^n w+1-a^n)=0$ for all $n\in\mathbb{Z}$. In particular $f$ cannot be analytic at $1$. So if $f$ is entire then it can have no zeros in $\mathbb{C}\setminus\{1\}$.
	 	\end{prop}
 	
	\begin{proof} We first note that the relation $f\circ\phi_{a}=\lambda f$ extends from $\mathbb{D}$ to all of $\mathbb{H}$ or $\mathbb{C}$ depending on whether $f$ is analytic on $\mathbb{H}$ or $\mathbb{C}$.  If $n\geq0$, then we get \[f(a^n w+1-a^n)=f(\phi_{a^n}(w))=\lambda^n f(w)=0.\] If $n<0$, then $\lambda^{-n}f(a^n w+1-a^n)=(C_{\phi_{a^{-n}}}f)(a^n w+1-a^n)=f(w)=0.$
	Hence $f(a^n w+1-a^n)=0$ for all $n\in\mathbb{Z}$.  If $f$ is analytic at $1$ then letting $n\to +\infty$ gives $f(1)=0$ and hence $f\equiv 0$. So $f$ is non-analytic at $1$. Therefore if $f$ is entire it must be zero-free in $\mathbb{C}\setminus\{1\}$.	\end{proof}

The following result shows that derivatives of eigenvectors are eigenvectors.

\begin{prop}\label{f'/f} If $C_{\phi_a}f=\lambda f$ for some $f\in\mathrm{Hol}(\mathbb{D})$ and $\lambda\in\mathbb{C}$, then $C_{\phi_a}f'=\frac{\lambda}{a} f'$. In particular, if all derivatives $f^{(n)}\in H^2$ for $n\in \mathbb{N}$ then $f$ must be a polynomial. 
	\end{prop}
\begin{proof}
	Derivating the relation $f(az+1-a)=\lambda f(z)$ with respect to $z$ implies $f'(az+1-a)=\frac{\lambda}{a}f'(z)$ and hence $C_{\phi_a}f'=\frac{\lambda}{a}f'$. It follows that $C_{\phi_a}f^{(n)}=\frac{\lambda}{a^n}f^{(n)}$. Now additionally suppose $f^{(n)}\in H^2$ for all $n\in \mathbb{N}$. Since $\lambda/a^n\to\infty$ as $n\to\infty$ and the spectrum of $C_{\phi_a}$ is a compact set, we must have $f^{(N)}\equiv 0$ for some $N\in\mathbb{N}$. Therefore $f$ is a polynomial. 
	\end{proof}

We have seen that the behaviour at the boundary point $1$ of eigenvectors plays an important role in our study. We next consider their radial limits at $1$. Recall that $f^*(1):=\lim_{r\to1^{-}}f(r)$ if it exists. For each $w\in\mathbb{D}$ define the \emph{orbit} of $w$ under $\phi_a$ by Orb$(w):=(\phi_{a^n}(w))_{n\in\mathbb{N}}$. It is clear that all orbits converge to $1$. Define the limit of $f$ along Orb$(w)$ as $w$-$\lim f=\lim_{n\to\infty}f(\phi_{a^n}(w))$. If $f^*(1)$ exists then $f^*(1)=w$-$\lim f$ for each $w\in(-1,1)$. The next result shows how the location of eigenvalues determines the radial limits of the corresponding eigenvectors. We recall that $\{\lambda\in\mathbb{C}:0<|\lambda|<a^{-1/2}\}$ are all eigenvalues of $C_{\phi_a}$ where $0<a<1$.

\begin{thm}\label{Boundary behaviour of eigenvectors}
	Let $f$ be an eigenvector for $C_{\phi_a}$ with eigenvalue $\lambda\in\mathbb{C}$. For $|\lambda|<1$ we have $f^*(1)=0$.  For $|\lambda|>1$ and $w\in(-1,1)$ we get
	\begin{equation*}
	w\text{-}\lim f=
	\begin{cases}
	0 & \mathrm{if} \ \ f(w)=0 \\
	\infty & \mathrm{if} \ \ f(w)\neq 0 \\
\end{cases}  
\end{equation*}
If $|\lambda|=1$, then $f^*(1)$ exists if and only if $f$ is a constant. 
\end{thm} 
\begin{proof} We first consider the case $|\lambda|<1$. Fix a point $r\in(-1,1)$ and consider the sequence of intervals 
	\[
	I_n:=[\phi_{a^n}(r),\phi_{a^{n+1}}(r)]
	 \]
	 for $n\geq 0$ with $I_0$ understood to denote $[r,\phi_a(r)]$. Then the interval 
	 $[r,1)=\bigcup_{n=0}^\infty I_n$ and $\phi_a(I_n)=I_{n+1}$. Let $C=\sup_{z\in I_0}|f(z)|$. Given $\epsilon>0$ small, there exists $N\in\mathbb{N}$ such that $|\lambda|^N<\epsilon$. Then for all $z\in[\phi_{a^N}(r),1)$ we have $z\in I_M$ for some $M\geq N$ and there exists $z_0\in I_0$ such that $\phi_{a^M}(z_0)=z$. So for all $z\in[\phi_{a^N}(r),1)$ we have
	 \[
	 |f(z)|=|f(\phi_{a^M}(z_0))|=|(C_{\phi_{a^M}}f)(z_0)|=|\lambda|^M|f(z_0)|\leq C\epsilon.
	 \]  
	 Since $\epsilon$ was arbitrary we get $f^*(1)=0$ if $|\lambda|<1$. Now suppose that  $|\lambda|>1$ and consider the equation 
	\begin{equation}\label{f(a^n)}
	f(\phi_{a^n}(w))=\lambda^n f(w)
	\end{equation}
	for $w\in(-1,1)$. Letting $n\to\infty$ in \eqref{f(a^n)} shows that $w\text{-}\lim f=0$ if $f(w)=0$ and $w\text{-}\lim f=\infty$ otherwise. If $\lambda=e^{i\theta}\neq 1$, then $\lambda^n$ does not converge and letting $n\to\infty$ in \eqref{f(a^n)} shows that $f^*(1)$ cannot exist. For $\lambda=1$ and $f$ non-constant, choose $w,s\in (-1,1)$ such that $f(w)\neq f(s)$ and hence \eqref{f(a^n)} implies 
	\[
	\lim_{n\to\infty}f(\phi_{a^n}(w))=f(w)\neq f(s)= \lim_{n\to\infty}f(\phi_{a^n}(s))
	\]
	so $f^*(1)$ does not exist. 
\end{proof}

Finally we prove that the $f_s$ are the only common eigenvectors for $C_{\phi_a}$ with $a\in(0,1)$. For any non-zero $f\in H^2$ we define the subset $A_f\subset(0,1)$ by
\[
A_f=\{a\in(0,1): f \ \mathrm{is} \ \mathrm{a} \ C_{\phi_a}\mathrm{eigenvector}\}.
\] 
Since $f_s$ is a common eigenvector for the $\{C_{\phi_a}:a\in(0,1)\}$ for all $\mathrm{Re}(s)>-1/2$ (see \eqref{Principal eigenvector}), it follows that $A_{f_s}=(0,1)$ . The next result shows that $A_f=(0,1)$ precisely when $f$ is a scalar multiple of one of these $f_s$. 

\begin{thm}\label{A_f Thm.}Let $f\in H^2$ not be a scalar multiple of $f_s$ for any $s\in\mathbb{C}$. Then either $A_f$ is empty or $A_f=(c^n)_{n\in\mathbb{N}}$ for some $c\in(0,1)$.
\end{thm}
\begin{proof} We first prove that $A_f$ is closed in $(0,1)$ and has empty interior. Let 
	\[\Omega=\{z\in\mathbb{D}:f(z)\neq 0\}\]
	be the open subset of $\mathbb{D}$ where $f$ is non-vanishing. Suppose $A_f$ in non-empty and there is a sequence $(b_n)_{n\in\mathbb{N}}$ in $A_f$ that converges to some $b\in (0,1)$ with $C_{\phi_{b_n}}f=\lambda_n f$. Then for $z\in\Omega$ we have 
	\[
	\lambda:=\lim_{n\to\infty}\lambda_n=\lim_{n\to\infty}\frac{C_{\phi_{b_n}}f(z)}{f(z)}=\lim_{n\to\infty}\frac{f(b_nz+1-b_n)}{f(z)}=\frac{C_{\phi_{b}}f(z)}{f(z)}
	\]
	which exists and hence $b\in A_f$. So $A_f$ is closed in $(0,1)$. Now suppose $A_f$ contains an open interval $(s,t)$. Then let $\lambda:(s,t)\to\mathbb{C}$ be the function defined by
	\begin{equation}\label{lambda(b)}
	f(bz+1-b)=C_{\phi_b}f(z)=\lambda(b)f(z).
	\end{equation}
	for $b\in(s,t)$. Then fixing $z\in\Omega$ in \eqref{lambda(b)} shows that $\lambda$ is continuously differentiable on $(s,t)$. Now differentiating \eqref{lambda(b)} with respect to $b$ while fixing $z$ gives
	\[
	f'(bz+1-b)=\frac{\lambda'(b)f(z)}{z-1}
	\]
	and doing the same with respect to $z$ while fixing $b$ gives
	\[
	f'(bz+1-b)=\frac{\lambda(b)f'(z)}{b}.
	\]
	Therefore $\frac{f'(z)}{f(z)}=\frac{-s(b)}{1-z}$ where $s(b)=\frac{b\lambda'(b)}{\lambda(b)}$ for all $z\in\Omega$ and $b\in(s,t)$. This implies $f\neq 0$ in $\mathbb{D}$ otherwise $f'/f$ would have a pole in $\mathbb{D}$. So $f$ has a holomorphic logarithm $g$ with $e^g=f$ in $\mathbb{D}$. Derivating the equation $fe^{-g}=1$ gives $f'e^{-g}=g'fe^{-g}=g'$ or $g'=f'/f$. Therefore $g(z)=s(b)\log(1-z)+C$ for a constant $C$ and hence $f(z)=K(1-z)^{s(b)}=Kf_{s(b)}$ for some constant $K$. This contradicts our assumption and hence $A_f$ has empty interior. 
	
	We now prove that $A_f=(c^n)_{n\in\mathbb{N}}$ for some $c\in(0,1)$ if $A_f\neq\emptyset$. Note that $a\in A_f$ implies $(a^n)_{n\in\mathbb{N}}\in A_f$ by \eqref{Compositional iterates}. If $a,b\in A_f$ with $C_{\phi_a}f=\lambda_af$ and $C_{\phi_b}f=\lambda_bf$, then $C_{\phi_{ab}}f=\lambda_a\lambda_bf$ and if $a<b$ then
	\[
	C_{\phi_{a/b}}f=\frac{1}{\lambda_b}C_{\phi_{a/b}}C_{\phi_b}f=\frac{1}{\lambda_b}C_{\phi_a}f=\frac{\lambda_a}{\lambda_b}f
	\] 
	where $\lambda_b$ is non-zero since $C_{\phi_b}$ is injective. So $a,b\in A_f$ implies $ab\in A_f$ and if also $a<b$ then $a/b\in A_f$. Now since the complement of $A_f$ in $(0,1)$ is open and dense, there exist $a,b\in A_f$ such that $(a,b)\cap A_f=\emptyset$. We define $c:=a/b\in A_f$ and hence $(c^n)_{n\in\mathbb{N}}\subset A_f$. Since $c>a$ and $c\notin(a,b)$ implies $c\geq b$. We now claim that $a=c^{N+1}$ and $b=c^N$ for some $N\in\mathbb{N}$. Otherwise there exists $n\in \mathbb{N}$ such that $c^{n+1}<a<b<c^n$ since $(a,b)\cap A_f=\emptyset$. But $b<c^n$ implies $a=cb<c^{n+1}$ which is a contradiction. So $a=c^{N+1}$ and $b=c^N$ for some $N\in\mathbb{N}$. This implies that for all $n\in\mathbb{N}$
	\begin{equation}\label{c_n cap A_f empty}
	(c^{n+1},c^n)\bigcap A_f=\emptyset
	\end{equation}
	otherwise if $d\in(c^{n+1},c^n)\bigcap A_f$ for some $n$ then $c^{N-n}d\in(a,b)\bigcap A_f$. The only case that remains is if $d>c$ and $d\in A_f$. But this is also not possible since $c>dc>c^2$ contradicts \eqref{c_n cap A_f empty} with $n=1$. Therefore $A_f=(c^n)_{n\in\mathbb{N}}$.
\end{proof}

\emph{Example}. Let $h=f_s+f_{s+\frac{2\pi i}{\log{a}}}$ for some $a\in(0,1)$ and $\mathrm{Re}(s)>-1/2$. Then $C_{\phi_a}h=a^sf_s+a^{s+\frac{2\pi i}{\log{a}}}f_{s+\frac{2\pi i}{\log{a}}}=a^sh$ because $a^{\frac{2\pi i}{\log{a}}}=1$ and hence $a\in A_h$. We will show that $A_h=(a^n)_{n\in\mathbb{N}}$. If some other $b\in A_h$ then \[C_{\phi_b}h=b^sf_s+b^{s+\frac{2\pi i}{\log{a}}}f_{s+\frac{2\pi i}{\log{a}}}=\lambda h\] for some $\lambda\in\mathbb{C}$ if and only if $b^{\frac{2\pi i}{\log{a}}}=e^{2\pi i\frac{\log{b}}{\log{a}}}=1$. So $\log{b}=n\log{a}=\log{a^n}$ for some $n\in\mathbb{Z}$. Hence $b=a^n$ for $n\geq 1$ since $b\in(0,1)$. Therefore $A_h=(a^n)_{n\in\mathbb{N}}$. \qed \smallskip

\section*{Acknowledgement}
This work constitutes a part of the doctoral thesis of the first author, which is supervised by the second author. The second author is partially supported by a FAPESP grant (17/09333-3).
	
\bibliographystyle{amsplain}

\end{document}